\documentclass[3p,twocolumn]{myelsarticle}

\usepackage[utf8]{inputenc}
\usepackage[T1]{fontenc}

\usepackage{amsmath}
\usepackage{amssymb}
\usepackage{amsthm}

\usepackage{tikz}

\usepackage{notationForne}

\usepackage{todonotes}

\usepackage[breaklinks,colorlinks,citecolor=blue,linkcolor=blue]{hyperref}

\newtheorem{theorem}{Theorem}
\newtheorem{lemma}[theorem]{Lemma} 
\newtheorem{observation}[theorem]{Observation} 
\theoremstyle{definition}

\newtheorem{example}[theorem]{Example}
\theoremstyle{remark}
\newtheorem{remark}[theorem]{Remark}

\begin{document}

\title{On the Complexity of Instationary Gas Flows}

\author[mg]{Martin Gro\ss}
\ead{martin.gross@mailbox.org}
\address[mg]{University of Waterloo, 200 University Ave W, Waterloo, ON N2L 3G1, Canada}

\author[mep]{Marc~E.~Pfetsch}
\ead{pfetsch@opt.tu-darmstadt.de}
\address[mep]{Research Group Optimization, TU Darmstadt, Dolivostr.\ 15, 64293 Darmstadt, Germany}

\author[ms]{Martin Skutella}
\ead{martin.skutella@tu-berlin.de}
\address[ms]{Technische Universit\"at Berlin, Stra\ss{}e des 17. Juni 136, 10623 Berlin, Germany}

\begin{abstract}
We study a simplistic model of instationary gas flows consisting
of a sequence of~$k$ stationary gas flows. We present efficiently
solvable cases and NP-hardness results, establishing complexity
gaps between stationary and instationary gas flows (already
for~$k=2$) as well as between instationary gas $s$-$t$-flows and
instationary gas $b$-flows.
\end{abstract}

\begin{keyword}
Gas transport network \sep instationary gas flow \sep time-dependent flow \sep complexity
\sep NP-hardness
\end{keyword}

\maketitle

\section{Introduction}

This paper studies the algorithmic complexity of time-varying flows
in gas transport networks. In the gas transport literature, these
flows are called \emph{instationary} in contrast to \emph{stationary}
gas flows that describe a steady state situation. This paper presents
efficiently solvable problems and 
identifies complexity gaps between stationary and instationary
gas flows, as well as between instationary gas flows with a single
source/sink and multi-terminal instationary gas flows. Our ultimate
goal is to contribute to a better understanding of the particular
difficulty of instationary gas flows. To this end, we introduce a 
simple model of instationary gas flows in Sect.~\ref{sec:model},
present an efficiently solvable instationary gas flow problem in
Sect.~\ref{sec:max-2-stage-flow}, examples of more complicated
scenarios in Sect.~\ref{sec:examples}, and finally an NP-hardness
result in Sect.~\ref{sec:complexity}.

\section{Stationary Gas Flows}
\label{sec:stationary}

Before turning to the more general case of instationary gas flows,
we introduce some basic facts
about stationary gas flows. In contrast to classical network flows
where, within given capacity bounds, flow may be distributed
throughout a network ad libitum, gas flows are governed by the laws
of physics. Essentially, in a gas network the (stationary) flow
along an arc (pipeline) is uniquely determined by the pressures
at the two endpoints of the arc. For an in-depth treatment of
(stationary) flows in gas networks we refer to the recent book~\cite{Koch_et_al:2015}.
The simplest and most widely adapted model for stationary gas flows is 
Weymouth's equation~\cite{Weymouth1912}: For an arc~$\arc=(\node,\otherNode)$,
the flow value~$\flow_\arc$ along~$\arc$ satisfies
\begin{align}
\beta_\arc\flow_\arc\abs{\flow_\arc} = \pot_\node - \pot_\otherNode,
\label{eq:flowpotential}
\end{align}
where the node potentials~$\pot_\node={p_\node}^2$ and $\pot_\otherNode={p_\otherNode}^2$
are the squared pressures at nodes~$\node$ and~$\otherNode$,
respectively, and~$\beta_\arc>0$ is a given constant specifying
the resistance of arc~$\arc$. Here, a negative flow value~$\flow_\arc$
on arc~$\arc=(\node,\otherNode)$ represents flow in the opposite
direction from node~$\otherNode$ to node~$\node$. This stationary gas flow
model forms the basis of this paper.

Consider a directed graph~$\graph$ with node set~$\nodes$ and arc
set~$\arcs$. For given node balances~$b\in\R^\nodes$ with
$\sum_{\otherNode\in\nodes} b_\otherNode=0$, a stationary gas flow satisfying
supplies and demands given by~$b$ can be computed by solving the
following convex min-cost $b$-flow problem~\cite{CCHKB78,Maugis1977}
\begin{align}\label{eq:ConvexProblem}
\begin{split}
\min&\quad\sum_{\arc\in\arcs}\frac{\beta_\arc}3\abs{\flow_\arc}^3\\
\text{s.t.}&\sum_{\arc\in\outArcs{\otherNode}}\flow_\arc
-\sum_{\arc\in\inArcs{\otherNode}}\flow_\arc=b_\otherNode
\quad\forall\otherNode\in\nodes,
\end{split}
\end{align}
with corresponding dual (strong duality holds)
\begin{equation}\label{eq:LagrangeDual}
\max_\pot\biggl(\sum_{\node\in\nodes}b_\node\pot_\node 
-2\sum_{(\node,\otherNode)\in\arcs} 
\frac{\abs{\pot_\node-\pot_\otherNode}^{3/2}}{3\sqrt{\beta_{(\node,\otherNode)}}}\biggr).
\end{equation}
The dual variables yield the node potentials
in~\eqref{eq:flowpotential}. These node potentials are
unique up to translation by an arbitrary value. The
problems~\eqref{eq:ConvexProblem} and~\eqref{eq:LagrangeDual} can be
solved efficiently within arbitrary precision.

Throughout this paper, we assume that there are uniform bounds
on all node potentials given by an interval~$[\pot_{\min},\pot_{\max}]$.
A stationary gas flow~$x$ with corresponding node
potential~$\pot\in\R^\nodes$ is \emph{feasible} if
$\pot_{\min}\leq\pot_\otherNode\leq\pot_{\max}$ for all~$\otherNode\in\nodes$.
Before introducing our model of instationary gas flows in the next section,
we state an important theorem on stationary gas flows, which essentially follows
from the work of Calvert and Keady~\cite{CalvertKeady1993}
(see also~\cite{GrossPfetschScheweEtAl2017}), and for which we give
a short proof for the sake of completeness.

\begin{theorem}[\cite{CalvertKeady1993}]\label{thm:Braess}
In a network with source~$s$, sink~$t$,
and potential interval~$[\pot_{\min},\pot_{\max}]$, the value of a maximal
feasible stationary gas $s$-$t$-flow cannot
be increased by increasing arc resistances~$\beta_\arc$, $\arc\in\arcs$. 
\end{theorem}

\begin{proof}
For a fixed $s$-$t$-flow value~$B=b_s=-b_t$ (and $b_\otherNode=0$
for~$\otherNode\in\nodes\setminus\{s,t\}$), consider the primal
problem~\eqref{eq:ConvexProblem} and the dual
problem~\eqref{eq:LagrangeDual}. Due to~$\eqref{eq:ConvexProblem}$,
the optimal value~$z^*$ of these two problems is an increasing
function of the arc resistances~$\beta_\arc$, $\arc\in\arcs$. Then,
by combining~\eqref{eq:LagrangeDual},~\eqref{eq:flowpotential},
and~\eqref{eq:ConvexProblem}, we obtain the following for the
optimal solution $(\flow^*,\pot^*)$:
\begin{align*}
z^* \!&=\! \sum_{\otherNode \in \nodes} b_\otherNode\, \pot^*_\otherNode 
- 2 \sum_{\arc \in \arcs} \frac{\beta_a}{3}|\flow^*_\arc|^3
\!=\! B\, (\pot^*_s-\pot^*_t) - 2z^*.
\end{align*}
For fixed $s$-$t$-flow value~$B$, the difference of
potentials at~$s$ and~$t$ is proportional to~$z^*$, more precisely $\pot^*_s-\pot^*_t= 3z^*/B$,
and thus an increasing function of the arc resistances~$\beta_\arc$,
$\arc\in\arcs$. Finally, the difference of potentials at~$s$ and~$t$
is also an increasing function of flow value~$B$ and bounded
by~$\pot_{\max}-\pot_{\min}$ where the maximum flow value is attained.
\end{proof}

\section{A Simple Instationary Gas Flow Model}
\label{sec:model}

We introduce a model of instationary gas flows that,
while being simple enough to allow for a theoretical analysis, 
still captures essential characteristics and exhibits interesting
properties. In particular, we prove meaningful results that
constitute an interesting first step in explaining the increased
difficulty of instationary versus stationary gas flows.

For $k\in\Z_{>0}$, a \emph{$k$-stage gas flow}~$\flow$ is a $k$-tuple
$(\flow^1,\dots,\flow^k)$ of stationary gas flows (where we interpret
$\flow^1,\dots,\flow^k$ as a temporal succession). If~$\flow^i$
satisfies supplies and demands~$b^i\in\R^\nodes$, $i=1,\dots,k$,
then~$\flow$ in total satisfies supplies and demands
$b=b^1+\cdots+b^k\in\R^\nodes$ and is called \emph{$k$-stage gas $b$-flow}.
For two distinguished nodes $s,t\in\nodes$,~$\flow$ is a
\emph{$k$-stage gas $s$-$t$-flow} of \emph{value}~$q$ if it satisfies
supplies and demands~$b\in\R^\nodes$ with $b_s=-b_t=q$ and
$b_\otherNode=0$ for $\otherNode\in\nodes\setminus\{s,t\}$.
A $k$-stage gas flow~$\flow$ is called \emph{stationary} if~$\flow^1=\dots=\flow^k$,
otherwise~$x$ is called \emph{instationary}. Finally, a
$k$-stage gas flow~$\flow$ is \emph{feasible} if
$\flow^1,\dots,\flow^k$ are feasible stationary gas flows.

\begin{remark}\label{remark:model}
Notice that, in marked contrast to actual gas transport, in our model
there is no correlation between consecutive flows~$x^i$ and~$x^{i+1}$.
Moreover, the model allows to arbitrarily buffer or borrow
flow in each node (\ie flow may be withdrawn or
injected at each node) at each stage as long as the accumulated node
balances~$b^1+\cdots+b^k$ add up to the desired~$b$ 
(cp.~examples in Sect.~\ref{sec:examples} below).
\end{remark}

We study the following two algorithmic problems for~$k\in\Z_{>0}$:
first, the \emph{maximum $k$-stage gas $s$-$t$-flow problem},
whose input is a network~$\graph$ with source~$s\in\nodes$,
sink~$t\in\nodes$, and interval $[\pot_{\min},\pot_{\max}]$,
and the task is to find a feasible $k$-stage gas $s$-$t$-flow
of maximum value; second, the \emph{$k$-stage gas $b$-flow problem},
whose input is a network~$\graph$ with supplies and
demands~$b\in\R^\nodes$, as well as interval~$[\pot_{\min},\pot_{\max}]$,
and the task here is to find a feasible $k$-stage gas $b$-flow.

\section{\texorpdfstring{Maximum $2$-Stage Gas $s$-$t$-Flows}{Maximum 2-Stage Gas s-t-Flows}}
\label{sec:max-2-stage-flow}

We first show that there exists an efficiently computable stationary
solution of the maximum $2$-stage gas $s$-$t$-flow problem.

\begin{theorem}\label{thm:max-2-stage-flow}
Taking two copies of the maximum feasible stationary
gas $s$-$t$-flow yields an optimal solution to the maximum $2$-stage gas $s$-$t$-flow problem.
\end{theorem}

In order to prove the theorem, we consider an arbitrary feasible
$2$-stage gas $s$-$t$-flow $(\flow^1,\flow^2)$ with corresponding
node potentials~$\pot^1,\pot^2\in\R^\nodes$. By definition, the
flow~$\tilde{\flow} \define \frac12(\flow^1+\flow^2)$ is an $s$-$t$-flow
(not necessarily a stationary gas flow, though), and the value of the
feasible $2$-stage gas $s$-$t$-flow $(\flow^1,\flow^2)$ is exactly
twice the value of~$\tilde{\flow}$.

\begin{lemma}\label{lem:average-potential-flow}
The node potentials $\bar\pot \define \frac12(\pot^1+\pot^2)$ induce a
feasible stationary gas flow~$\bar{\flow}$ with
$\sgn(\bar{\flow}_\arc)=\sgn(\tilde{\flow}_\arc)$ and $\abs{\bar{\flow}_\arc}\geq\abs{\tilde{\flow}_\arc}$
for each~$\arc\in\arcs$.
\end{lemma}

\begin{proof}
By definition of~$\bar{\flow}$ and~$\flow^i$, $i=1,2$, we have
\begin{align*}
\bar{\flow}_\arc &=
\sgn(\bar\pot_\node-\bar\pot_\otherNode)
\sqrt{\abs{\bar\pot_\node-\bar\pot_\otherNode}}/\sqrt{\beta_\arc},\\
\flow^i_\arc &=\sgn(\pot^i_\node-\pot^i_\otherNode)
\sqrt{\abs{\pot^i_\node-\pot^i_\otherNode}}/\sqrt{\beta_\arc},
\end{align*}
for each arc $\arc=(\node,\otherNode)\in\arcs$. Moreover, by
definition of~$\bar\pot$, we get
$\bar\pot_\node-\bar\pot_\otherNode
=\bigl((\pot^1_\node-\pot^1_\otherNode)+(\pot^2_\node-\pot^2_\otherNode)\bigr)/2$.
The lemma thus follows from the next observation.
\end{proof}

\begin{observation}
Consider the function~$f:\R\to\R$ with 
$f(\sigma)=\sgn(\sigma)\sqrt{\abs{\sigma}}$. Then, for all~$\sigma^1$, $\sigma^2\in\R$,
\begin{align*}
\sgn\Big(f\Bigl(\frac{\sigma^1+\sigma^2}2\Bigr)\Big)
&=\sgn\left(\frac{f(\sigma^1)+f(\sigma^2)}2\right),\\
\Big\lvert f\Bigl(\frac{\sigma^1+\sigma^2}2\Bigr)\Big\rvert
&\geq\bigg\lvert\frac{f(\sigma^1)+f(\sigma^2)}2\bigg\rvert.
\end{align*}
\end{observation}
 
\begin{proof}
Notice that~$f(-\sigma)=-f(\sigma)$ for all~$\sigma\in\R$ (in 
particular, $f(0)=0$), and $f|_{\R_{\geq0}}$ is non-negative,
strictly increasing, and concave. Therefore the statement is clear
for the cases that~$\sigma^1$ and~$\sigma^2$ are both non-negative or
both non-positive. 

It remains to consider the case $\sigma^1<0<\sigma^2$. The equality
statement on the signs is an immediate consequence of~$f$'s
properties noted above. By symmetry we may assume that 
$\abs{\sigma^1}\leq\sigma^2$ such that $\frac12(\sigma^1+\sigma^2)\geq0$
and thus $f\bigl(\frac12(\sigma^1+\sigma^2)\bigr)\geq0$. By concavity
of $f|_{\R_{\geq0}}$, we get two inequalities:
\begin{align*}
&f\Bigl(\frac{\sigma^1+\sigma^2}2\Bigr)
\geq \frac{f(0)+f(\sigma^1+\sigma^2)}2
=\frac{f(\sigma^1+\sigma^2)}2,
\end{align*}
\vspace{-5ex}
\begin{align*}
f(\sigma^1+\sigma^2)-f(\sigma^1)
&=f(-\abs{\sigma^1}+\sigma^2)+f(\abs{\sigma^1})\\
&\geq f(\sigma^2)+f(0)=f(\sigma^2).
\end{align*}
The latter inequality implies that
$f(\sigma^1+\sigma^2)\geq f(\sigma^1)+f(\sigma^2)$. Together with the
former inequality this yields the desired result.
\end{proof}

It follows from Lemma~\ref{lem:average-potential-flow} and~\eqref{eq:flowpotential} that by
increasing the~$\beta_\arc$ values individually for each 
arc~$\arc\in\arcs$, we arrive at a network where the 
$s$-$t$-flow~$\tilde\flow$ is a feasible stationary gas
$s$-$t$-flow induced by the node potentials
$\bar\pot \define \frac12(\pot^1+\pot^2)$.
More precisely, we need to set
$\tilde\beta_\arc
\define \beta_\arc \bar{\flow}_\arc^2/\tilde{\flow}_\arc^2
\geq\beta_\arc$.
Thus, by Theorem~\ref{thm:Braess}, the value of the stationary 
maximal feasible gas $s$-$t$-flow~$x^*$ in the network with original
values~$\beta_\arc$, $\arc\in\arcs$, is at least the value 
of~$\tilde\flow$, which is half the value of our feasible $2$-stage
gas $s$-$t$-flow~$(\flow^1,\flow^2)$. Summarizing, the value of the
feasible $2$-stage gas $s$-$t$-flow~$(x^*,x^*)$ is at least the value
of~$(\flow^1,\flow^2)$. This concludes the proof of 
Theorem~\ref{thm:max-2-stage-flow}.

\section{Examples and Counterexamples}
\label{sec:examples}

In this section we show that Theorem~\ref{thm:max-2-stage-flow} can
neither be generalized to the $k$-stage gas $s$-$t$-flow problem
for~$k\geq3$ nor to the $2$-stage gas $b$-flow problem.

\subsection{\texorpdfstring{Instationary $k$-stage gas $s$-$t$-flows for $k\geq3$}{Instationary k-stage gas s-t-flows for k >= 3}}

We present a network for the maximum $3$-stage gas $s$-$t$-flow problem where repeating the maximum feasible stationary
gas $s$-$t$-flow three times is not optimal. In order to develop the right
intuition for this instance, we first show that fixing the potentials
of nodes~$s$ and~$t$ to the same value does not keep us from sending
a positive amount of flow in a $3$-stage gas $s$-$t$-flow.

\begin{example}\label{ex:pump}
Consider a path network with nodes $\nodes=\{s,\node,\otherNode,t\}$,
arcs~$\arcs=\{(s,\node),(\node,\otherNode),(\otherNode,t)\}$, and
$\beta_\arc=1$, for all~$\arc\in\arcs$. Moreover,
$[\pot_{\min}, \pot_{\max}] = [0,4]$. In Fig.~\ref{fig:pump} we
present the potentials of a $3$-stage gas $s$-$t$-flow of
value~$2-\sqrt{2}>0$, where the potentials of~$s$ and~$t$ are fixed
to~$2$.
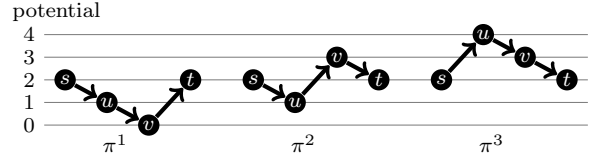
\begin{figure}[t]
\centering
\begin{tikzpicture}[auto,xscale=0.55,yscale=0.75]
\newcommand{\scale}{0.4}
\tikzstyle{node}=[circle,inner sep = 0pt,minimum size=0.8em,fill,color=black,text=white]
\tikzstyle{arc}=[ultra thick,->]
\foreach \i in {0,1,2,3,4} \draw [gray] (-0.5,\scale*\i) node [black,left] {\footnotesize$\i$} -- +(13,0);
\node at (-0.2,5*\scale) {\footnotesize potential};

\begin{scope}
\node [node] (s) at (0,2*\scale) {\footnotesize$s$};	
\node [node] (u) at (1,1*\scale) {\footnotesize$\node$};	
\node [node] (v) at (2,0*\scale) {\footnotesize$\otherNode$};	
\node [node] (t) at (3,2*\scale) {\footnotesize$t$};	
\draw [arc] (s) -- (u);
\draw [arc] (u) -- (v);
\draw [arc] (v) -- (t);
\node at (1.2,-0.8*\scale) {\footnotesize$\pot^1$};
\end{scope}

\begin{scope}[xshift=4.5cm]
\node [node] (s) at (0,2*\scale) {\footnotesize$s$};	
\node [node] (u) at (1,1*\scale) {\footnotesize$\node$};	
\node [node] (v) at (2,3*\scale) {\footnotesize$\otherNode$};	
\node [node] (t) at (3,2*\scale) {\footnotesize$t$};	
\draw [arc] (s) -- (u);
\draw [arc] (u) -- (v);
\draw [arc] (v) -- (t);
\node at (1.2,-0.8*\scale) {\footnotesize$\pot^2$};
\end{scope}

\begin{scope}[xshift=9cm]
\node [node] (s) at (0,2*\scale) {\footnotesize$s$};	
\node [node] (u) at (1,4*\scale) {\footnotesize$\node$};	
\node [node] (v) at (2,3*\scale) {\footnotesize$\otherNode$};	
\node [node] (t) at (3,2*\scale) {\footnotesize$t$};	
\draw [arc] (s) -- (u);
\draw [arc] (u) -- (v);
\draw [arc] (v) -- (t);
\node at (1.2,-0.8*\scale) {\footnotesize$\pot^3$};
\end{scope}
\end{tikzpicture}		
\caption{Node potentials of $3$-stage gas $s$-$t$-flow with
value $2-\sqrt{2}$ on the path network described in
Example~\ref{ex:pump}}
\label{fig:pump}
\end{figure}
Note that, on every arc, there is flow of value~$1$ in two of
the three stages and flow of value~$-\sqrt{2}$ in the remaining 
stage. In particular, the individual stationary gas flows of the
three stages are not $s$-$t$-flows but use the model's freedom
to buffer flow at intermediate nodes~$\node$ and~$\otherNode$
(cf.~Remark~\ref{remark:model}).
\end{example}

In Sect.~\ref{sec:max-2-stage-flow} we have turned a given $2$-stage
gas flow into a stationary gas flow by considering the average node
potentials~$\bar\pot$. Notice that this idea is completely useless
with respect to Example~\ref{ex:pump}. The average potential of any
node in the given $3$-stage gas flow is equal to~$2$
(cf.~Fig.~\ref{fig:pump}). In particular, $\bar\pot$ induces the (stationary)
zero flow.

In the next example, we use the intuition behind
Example~\ref{ex:pump} as a gadget to come up with a path network
where any maximum $k$-stage gas $s$-$t$-flow is instationary. More
precisely, we extend the path by adding two additional nodes, one on
the left and one on the right, together with arcs of high
resistance connecting them to the
corresponding ends of the previous path.

\begin{example}\label{ex:instationary-max-flow}
Consider a path network consisting of node 
set~$\nodes=\{s,s',\node,\otherNode,t',t\}$ and arc 
set~$\arcs=\{(s,s'),(s',\node),(\node,\otherNode),(\otherNode,t'),(t',t)\}$
with $\beta_\arc=1$ for all arcs~$\arc$, except $\beta_{(s,s')}=\beta_{(t',t)}=27+18\sqrt{2}$.
Moreover, $[\pot_{\min}, \pot_{\max}] = [0,4]$ as in Example~\ref{ex:pump}. In 
Fig.~\ref{fig:max-stat-s-t-flow}, we give a maximum feasible
stationary gas $s$-$t$-flow of value~$\approx0.193$.
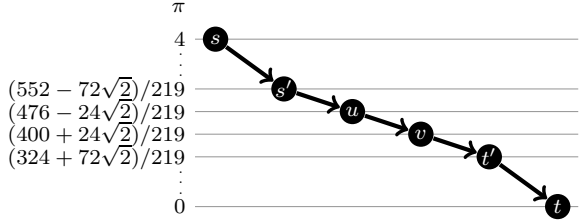
\begin{figure}[tb]
\centering
\begin{tikzpicture}[auto,xscale=0.9]
\newcommand{\scale}{0.3}
\tikzstyle{node}=[circle,inner sep = 0pt,minimum size=0.95em,fill,color=black,text=white]
\tikzstyle{arc}=[ultra thick,->]
\foreach \i/\l in {0.8/0,3/{(324+72\sqrt{2})/219},4/{(400+24\sqrt{2})/219},5/{(476-24\sqrt{2})/219},6/{(552-72\sqrt{2})/219},8.2/4}
	\draw [gray!80!white] (-0.3,\scale*\i) node [black,left] {\footnotesize$\l$} -- +(5.6,0);
\draw (-0.3,9.6*\scale) node [left] {\footnotesize$\pot$};
\draw (-0.33,2.2*\scale) node [left] {\tiny$\vdots$};
\draw (-0.33,7.4*\scale) node [left] {\tiny$\vdots$};

\node [node] (s) at (0,8.2*\scale) {\footnotesize $s$};
\node [node] (sp) at (1,6*\scale) {\footnotesize $s'$};	
\node [node] (u) at (2,5*\scale) {\footnotesize $\node$};	
\node [node] (v) at (3,4*\scale) {\footnotesize $\otherNode$};	
\node [node] (tp) at (4,3*\scale) {\footnotesize $t'$};	
\node [node] (t) at (5,0.8*\scale) {\footnotesize $t$};
\draw [arc] (s) -- (sp);
\draw [arc] (sp) -- (u);
\draw [arc] (u) -- (v);
\draw [arc] (v) -- (tp);
\draw [arc] (tp) -- (t);
\end{tikzpicture}		
\caption{Node potentials of a maximum feasible stationary gas
$s$-$t$-flow of
value~$\smash{\sqrt{\raisebox{0ex}[2ex][0ex]{\ensuremath{(76-48\sqrt{2})/219}}}}\approx0.193$
on the path network described in
Example~\ref{ex:instationary-max-flow}}
\label{fig:max-stat-s-t-flow}
\end{figure}
Repeating this flow three times yields a feasible $3$-stage gas 
$s$-$t$-flow of value~$\approx0.578$. There is, however, an
instationary solution achieving value~$2-\sqrt{2}\approx0.586$
which can be achieved as follows. Fix the potentials of node~$s$
to~$4$ and of node~$t$ to~$0$; for the remaining `inner' nodes, plug
in the solution from Example~\ref{ex:pump}, that is, fix the 
potentials of~$s'$ and~$t'$ to~$2$ and let the potentials of~$\node$
and~$\otherNode$ vary as in Fig.~\ref{fig:pump}.
\end{example}
 
\begin{remark}
The gap between the instationary solution and the optimal stationary
solution in Example~\ref{ex:instationary-max-flow} is apparently
tiny. With a simple trick we can, however, construct path networks
where the value of an instationary $3$-stage gas $s$-$t$-flow exceeds
the value of any stationary $3$-stage gas $s$-$t$-flow by an
arbitrarily large factor. Such networks can be obtained by replacing
the $s'$-$t'$-subnetwork in Example~\ref{ex:instationary-max-flow} by
a serial composition of~$\ell$ copies of this subnetwork. It is not
difficult to see that the value of a maximum feasible stationary gas
$s$-$t$-flow tends to zero when~$\ell$ tends to infinity. On the other
hand, the instationary $3$-stage gas $s$-$t$-flow of
value~$2-\sqrt{2}$ described in Example~\ref{ex:instationary-max-flow}
can be extended to the larger network by operating each of the~$\ell$
copies as depicted in Fig.~\ref{fig:pump}.
\end{remark}

Finally notice that the examples and results for the maximum $3$-stage gas
$s$-$t$-flow problem discussed in this section can be generalized in
a straightforward way to $k$-stages for~$k>3$.

\subsection{\texorpdfstring{Instationary $2$-stage gas $b$-flows}{Instationary 2-stage gas b-flows}}

We present a network with supplies and demands~$b\in\R^V$, where any
stationary $k$-stage gas $b$-flow requires a considerably larger
interval of node potentials than an instationary $k$-stage
gas $b$-flow.

\begin{example}\label{ex:b-flow}
For some fixed parameter~$0<\varepsilon<1$ consider the path network with
$2q+2$ nodes~$V=\{\node_0,\otherNode_0,\node_1,\otherNode_1,\dots,\node_q,\otherNode_q\}$
depicted in Fig.~\ref{fig:b-flow}.
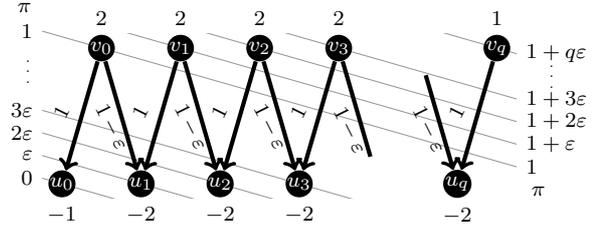
\begin{figure}[tb]
\centering
\begin{tikzpicture}[auto,xscale=0.52,yscale=0.9]
\tikzstyle{node}=[circle,inner sep = 0pt,minimum size=1em,fill,color=black,text=white]
\tikzstyle{arc}=[ultra thick,<-]

\begin{scope}
\clip (-1.4,-0.22) -- (-1.4,2.4) -- (0,2.4) -- (0,2.22) -- (13.35,2.22) -- (13.35,-0.22) -- cycle;
\foreach \i/\l in {0/0,1/\varepsilon,2/2\varepsilon,3/3\varepsilon,6.5/1,7.5/1+\varepsilon,8.5/1+2\varepsilon,9.5/1+3\varepsilon,11.5/1+q\varepsilon}
\draw [gray!80!white] (-0.5,1/12+4/12*\i) node [black,left] {\footnotesize$\l$} -- +(12,-2) node [black,right] {\footnotesize$\l$};
\draw (12.58-0.5,-1.89+1/12+4/12*10.5) node [right] {\tiny$\vdots$};
\draw (-0.55,0.1+1/12+4/12*4.75) node [left] {\tiny$\vdots$};
\end{scope}
\node at (-0.95,2.6) {\footnotesize$\pot$};
\node at (12.05,-0.1) {\footnotesize$\pot$};

\node [node,label=below:\footnotesize$-1$] (u0) at (2*0,0) {\footnotesize$\node_0$};
\node [node,label=above:\footnotesize$2$] (v0) at (1+2*0,2) {\footnotesize$\otherNode_0$};
\draw [arc] (u0) -- node [sloped,pos=0.6] {\footnotesize$1$} (v0);

\foreach \i in {1,2,3}{
\node [node,label=below:\footnotesize$-2$] (u\i) at (2*\i,0) {\footnotesize$\node_{\i}$};
\node [node,label=above:\footnotesize$2$] (v\i) at (1+2*\i,2) {\footnotesize$\otherNode_{\i}$};
\draw [arc] (u\i) -- node [sloped,pos=0.6] {\footnotesize$1$} (v\i);
}
\draw [arc] (u1) -- node [sloped,pos=0.2,xshift=1mm,yshift=1.5mm] {\footnotesize$1-\varepsilon$} (v0);
\draw [arc] (u2) -- node [sloped,pos=0.2,xshift=1mm,yshift=1.5mm] {\footnotesize$1-\varepsilon$} (v1);
\draw [arc] (u3) -- node [sloped,pos=0.2,xshift=1mm,yshift=1.5mm] {\footnotesize$1-\varepsilon$} (v2);
\draw [arc,-] (v3) -- node [swap,sloped,pos=0.95,xshift=1mm,yshift=1.5mm] {\footnotesize$1-\varepsilon$} +(0.8,-1.6);

\node [node,label=below:\footnotesize$-2$] (uq) at (2*5,0) {\footnotesize$\node_{q}$};
\node [node,label=above:\footnotesize$1$] (vq) at (1+2*5,2) {\footnotesize$\otherNode_{q}$};
\draw [arc] (uq) -- node [sloped,pos=0.6] {\footnotesize$1$} (vq);
\draw [arc,<-] (uq) -- node [sloped,pos=0.2,xshift=1mm,yshift=1.5mm] {\footnotesize$1-\varepsilon$} +(-0.8,1.6);
\end{tikzpicture}
\caption{Node potentials $\pot\in\R^\nodes$ inducing a stationary gas $b/2$-flow in the path network of
Example~\ref{ex:b-flow};
the numbers at arcs indicate the $\beta_\arc$-values,
the numbers at nodes $b_\node$-values.}
\label{fig:b-flow}
\end{figure}
There are arcs~$(\otherNode_i,\node_i)$ with
$\beta_{(\otherNode_i,\node_i)}=1$ for~$i=0,\dots,q$ and
arcs~$(\otherNode_i,\node_{i+1})$ for~$i=0,\dots,q-1$
with~$\beta_{(\otherNode_i,\node_{i+1})}=1-\varepsilon$.
The supplies and demands are~$b(\node_0)=-2$, $b(\node_i)=-4$
for~$i=1,\dots,q$, $b(\otherNode_i)=4$ for~$i=0,\dots,q-1$,
and~$b(\otherNode_q)=2$.
The stationary gas flow induced by the potentials in Fig.~\ref{fig:b-flow}
sends one unit of flow along each arc and thus fulfills
supplies and demands~$b/2$. It therefore yields the unique stationary $2$-stage
gas $b$-flow, and its range of node potentials is~$1+q\varepsilon$.

In Fig.~\ref{fig:b-flow2}, we present node potentials inducing an instationary $2$-stage gas
$b$-flow.
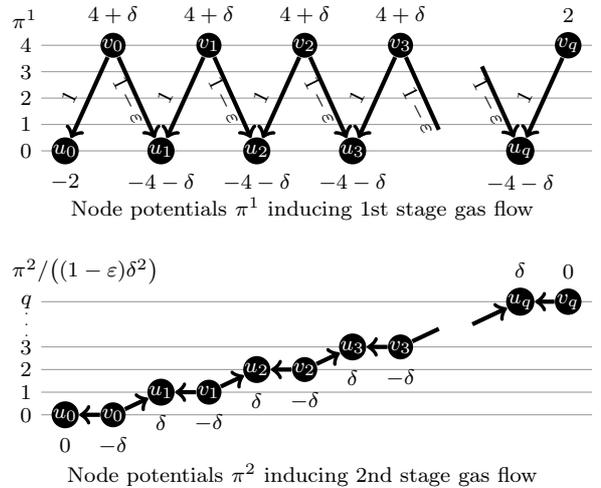
\begin{figure}[bt]
\centering
\begin{tikzpicture}[auto,xscale=0.63]
\tikzstyle{node}=[circle,inner sep = 0pt,minimum size=0.9em,fill,color=black,text=white]
\tikzstyle{arc}=[ultra thick,<-]

\begin{scope}[yscale=0.7]
	
\node at (4.95,-1.1) {\footnotesize Node potentials $\pot^1$ inducing 1st stage gas flow};

\foreach \i in {0,...,4}
	\draw [gray!80!white] (-0.5,0.5*\i) node [black,left] {\footnotesize$\i$} -- +(11.5,0);

\node[anchor=west] at (-1.3,2.5) {\footnotesize$\pot^1$};

\node [node,label=below:\footnotesize$-2$] (u0) at (2*0,0) {\footnotesize$\node_0$};
\node [node,label=above:{\footnotesize$4+\delta$}] (v0) at (1+2*0,2) {\footnotesize$\otherNode_0$};
\draw [arc] (u0) -- node [sloped,pos=0.6] {\footnotesize$1$} (v0);

\foreach \i in {1,...,3}{
\node [node,label=below:{\footnotesize$-4-\delta$}] (u\i) at (2*\i,0) {\footnotesize$\node_{\i}$};
\node [node,label=above:{\footnotesize$4+\delta$}] (v\i) at (1+2*\i,2) {\footnotesize$\otherNode_{\i}$};
\draw [arc] (u\i) -- node [sloped,pos=0.6] {\footnotesize$1$} (v\i);
}
\draw [arc] (u1) -- node [sloped,pos=0.2,xshift=1mm,yshift=1.5mm] {\footnotesize$1-\varepsilon$} (v0);
\draw [arc] (u2) -- node [sloped,pos=0.2,xshift=1mm,yshift=1.5mm] {\footnotesize$1-\varepsilon$} (v1);
\draw [arc] (u3) -- node [sloped,pos=0.2,xshift=1mm,yshift=1.5mm] {\footnotesize$1-\varepsilon$} (v2);
\draw [arc,-] (v3) -- node [swap,sloped,pos=0.55,xshift=6mm,yshift=1.0mm] {\footnotesize$1-\varepsilon$} +(0.8,-1.6);

\node [node,label=below:{\footnotesize$-4-\delta$}] (uq) at (2*5-0.5,0) {\footnotesize$\node_{q}$};
\node [node,label=above:\footnotesize$2$] (vq) at (2*5+0.5,2) {\footnotesize$\otherNode_{q}$};
\draw [arc] (uq) -- node [sloped,pos=0.6] {\footnotesize$1$} (vq);
\draw [arc,<-] (uq) -- node [sloped,pos=0.2,xshift=1mm,yshift=1.5mm] {\footnotesize$1-\varepsilon$} +(-0.8,1.6);	
\end{scope}

\begin{scope}[yshift=-3.5cm]
\foreach \i/\l in {0/0,1/1,2/2,3/3,5/q}
\draw [gray!80!white] (-0.5,0.3*\i) node [black,left] {\footnotesize$\l$} -- +(11.5,0);

\draw (-0.55,0.1+0.3*4) node [left] {\tiny$\vdots$};

\node at (4.95,-0.8) {\footnotesize Node potentials $\pot^2$ inducing 2nd stage gas flow};

\node[anchor=west] at (-1.3,1.9) {\footnotesize$\pot^2/\bigl((1-\varepsilon)\delta^2\bigr)$};

\begin{scope}
\node [node,label=below:\footnotesize$0$] (u0) at (0,0) {\footnotesize$\node_0$};
\node [node,label=below:{\footnotesize$-\delta$}] (v0) at (1,0) {\footnotesize$\otherNode_0$};
\draw [arc] (u0) -- (v0);

\foreach \i in {1,...,3}{
\node [node,label=below:{\footnotesize$\delta$}] (u\i) at (2*\i,0.3*\i) {\footnotesize$\node_{\i}$};
\node [node,label=below:{\footnotesize$-\delta$}] (v\i) at (1+2*\i,0.3*\i) {\footnotesize$\otherNode_{\i}$};
\draw [arc] (u\i) -- (v\i);
}
\draw [arc] (u1) -- (v0);
\draw [arc] (u2) -- (v1);
\draw [arc] (u3) -- (v2);
\draw [arc,-] (v3) -- +(0.8,0.24);

\node [node,label=above:{\footnotesize$\delta$}] (uq) at (2*5-0.5,5*0.3) {\footnotesize$\node_{q}$};
\node [node,label=above:\footnotesize$0$] (vq) at (2*5+0.5,5*0.3) {\footnotesize$\otherNode_{q}$};
\draw [arc] (uq) -- (vq);
\draw [arc,<-] (uq) -- +(-1,-0.3);
\end{scope}
\end{scope}

\end{tikzpicture}
\caption{Node potentials~$\pot^1$ and $\pot^2$ inducing an
instationary $2$-stage gas $b$-flow in
the network of Example~\ref{ex:b-flow}; here
$\delta \define 2/\sqrt{1-\varepsilon}-2\in\theta(\varepsilon)$
for~$\varepsilon\to0$.}
\label{fig:b-flow2}
\end{figure}	
Notice that, by choice of the node potentials~$\pot^1$, the
first-stage gas flow~$x^1$ overfulfills the supplies at
nodes~$\otherNode_0,\dots,\otherNode_{q-1}$ and the demands at
nodes~$\node_1,\dots,\node_q$ slightly by $\delta \define 2/{\sqrt{1-\varepsilon}}-2\in\theta(\varepsilon)$ for~$\varepsilon\to0$.
This is compensated for in the second stage (cf.~Remark~\ref{remark:model}). Overall, the range of
node potentials has size~$\max\{4,q(1-\varepsilon)\delta^2\}$ (see Fig.~\ref{fig:b-flow2}). If
we set~$\varepsilon \define 1/\sqrt{q}$ and let~$q$ go to infinity, the
node potentials are bounded by a constant. For the stationary
$2$-stage gas $b$-flow (see Fig.~\ref{fig:b-flow}), however, the size
of the range of node potentials is~$1+\sqrt{q}$ and thus unbounded.
\end{example}

\section{Complexity Results}
\label{sec:complexity}

We finally prove the following hardness result.

\begin{theorem}
\label{thm:NP-hard}
For a given network with supplies and demands~$b\in\R^V$ and potential interval~$[\pot_{\min},\pot_{\max}]$,
it is strongly NP-hard to decide whether there exists a feasible $2$-stage gas $b$-flow.
\end{theorem}

In order to prove this result, we first introduce several gadgets,
using our insights from Sect.~\ref{sec:max-2-stage-flow}.

\subsection{Nodes with fixed potential}
\label{subsec:fixed-potential}

Our first gadget is used to fix the potential of some node~$\node$
in any feasible $2$-stage gas $b$-flow to a given 
value~$\pot^*_\node$ with $\pot_{\min}<\pot^*_\node<\pot_{\max}$. To
this end, we introduce two new nodes~$\node^s$ and~$\node^t$ whose
only incident arcs are~$\arc^s=(\node^s,\node)$ 
and~$\arc^t=(\node,\node^t)$, respectively; see 
Fig.~\ref{fig:gadget1}. 
\begin{figure}[t]
\centering
\begin{tikzpicture}[auto,xscale=0.75]
\tikzstyle{node}=[circle,inner sep = 0pt,minimum size=1em,fill,color=black,text=white]
\tikzstyle{arc}=[ultra thick,->]
\node [node] (u) at (0,0) {\footnotesize$\node$};	
\node [node,label=above:\footnotesize${b_{\node^s}=2}$] (s) at (-4,0) {\footnotesize$\node^s$};	
\node [node,label=above:\footnotesize${b_{u^t}=-2}$] (t) at (4,0) {\footnotesize$\node^t$};
\draw [arc] (s) to node {\footnotesize$\arc^s$} node [swap] {\footnotesize$\beta_{\arc^s}=\pot_{\max}-\pot^*_\node$} (u);
\draw [arc] (u) to node {\footnotesize$\arc^t$} node [swap] {\footnotesize$\beta_{\arc^t}=\pot^*_\node-\pot_{\min}$} (t);
\end{tikzpicture}
\caption{Gadget fixing $\node$'s potential to given value~$\pot^*_\node$}
\label{fig:gadget1}
\end{figure}
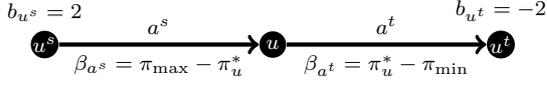
Moreover, we set
$\beta_{\arc^s} \define \pot_{\max}-\pot^*_\node$, $b_{\node^s} \define 2$,
$\beta_{\arc^t} \define \pot^*_\node-\pot_{\min}$, and $b_{\node^t} \define -2$.
By construction, the supply and demand at~$\node^s$ and~$\node^t$,
respectively, can be satisfied by a $2$-stage gas flow if the node
potentials are set to $\pot^1_{\node^s}=\pot^2_{\node^s}=\pot_{\max}$, $\pot^1_\node=\pot^2_\node=\pot^*_\node$, and
$\pot^1_{\node^t}=\pot^2_{\node^t}=\pot_{\min}$.

\begin{lemma}
In any feasible $2$-stage gas flow satisfying the supply and demand
at~$\node^s$ and~$\node^t$, respectively, node $\node$'s potential
satisfies $\pot^1_\node=\pot^2_\node=\pot^*_\node$.
\end{lemma}

\begin{proof}
In order to satisfy the supply at~$\node^s$, the total flow on
arc~$\arc^s$ must sum up to~$b_{\node^s}=2$. Thus, the node potentials~$\pot^1$
and~$\pot^2$ must satisfy
\begin{align}
\begin{split}
2&=\flow^1_{\arc^s}+\flow^2_{\arc^s}
\leq\frac{\sqrt{\pot_{\max}-\pot^1_\node}
+\sqrt{\pot_{\max}-\pot^2_\node}}{\sqrt{\pot_{\max}-\pot^*_\node}}\\
&\leq2\frac{\sqrt{\pot_{\max}-(\pot^1_\node+\pot^2_\node)/2}}{\sqrt{\pot_{\max}-\pot^*_\node}}.
\end{split}
\label{eq:lem_lower_bd}
\end{align}
Here, the first inequality holds since the flow on arc~$\arc^s$ is
maximal if~$\pot^1_{\node^s}=\pot^2_{\node^s}=\pot_{\max}$. The
second inequality follows from the concavity of the square root
function. Notice that, in order for the right hand side expression
to be at least~$2$, the average
potential~$(\pot^1_\node+\pot^2_\node)/2$ must not
exceed~$\pot^*_\node$. Using an analogous argument for the total
flow on arc~$\arc^t$, it can be shown that 
the average
potential~$(\pot^1_\node+\pot^2_\node)/2$ must not fall 
below~$\pot^*_\node$. Thus the aver\-age
potential must equal~$\pot^*_\node$. As a con\-se\-quen\-ce, equality
holds in~\eqref{eq:lem_lower_bd}
which, by
strict concavity of the square root function,
implies~$\pot^1_\node=\pot^2_\node=\pot^*_\node$.
\end{proof}

\subsection{Binary decision nodes}
\label{subsec:binary-gadget}

Our second gadget is used to create a node~$\node$ to model a binary
decision. More precisely, there are two possibilities: either~$\pot^1_\node=\pot^2_\node=1$ or it must attain the two 
values~$\pot_{\min}=0$ and~$\pot_{\max}=4$, that is, $\{\pot^1_\node,\pot^2_\node\}=\{0,4\}$.
With this end in view, we introduce two additional nodes~$\otherNode$
and~$\thirdNode$ with fixed potentials~$\pot^*_\otherNode=0$,
$\pot^*_\thirdNode=\frac45$ and balances~$b_\otherNode=-2$
and~$b_\thirdNode=-2/\sqrt{5}$. Moreover, nodes~$\otherNode$ 
and~$\thirdNode$ are connected to~$\node$ by the
arcs~$\arc^\otherNode=(\node,\otherNode)$ 
and~$\arc^\thirdNode=(\node,\thirdNode)$
with~$\beta_{\arc^\otherNode}=\beta_{\arc^\thirdNode}=1$. Finally, we
set~$b_\node \define -(b_\otherNode+b_\thirdNode)=2+2/\sqrt{5}$; see 
Fig.~\ref{fig:gadget2}.
\begin{figure}[t]
\centering
\begin{tikzpicture}[auto,xscale=0.9,yscale=0.75]
\tikzstyle{node}=[circle,inner sep = 0pt,minimum size=1em,fill,color=black,text=white]
\tikzstyle{arc}=[ultra thick,->]
\node [node,label=left:\footnotesize${b_\node=2+2/\sqrt{5}}$] (u) at (0,0) {\footnotesize$\node$};	
\node [node,label=right:\footnotesize${b_{\otherNode}=-2}$,label=above:\footnotesize${\pot^*_\otherNode=0}$] (v) at (3,1) {\footnotesize$\otherNode$};	
\node [node,label=right:\footnotesize${b_{\thirdNode}=-2/\sqrt{5}}$,label=above:\footnotesize${\pot^*_\thirdNode=\frac45}$] (w) at (3,-1) {\footnotesize$\thirdNode$};
\draw [arc] (u) to node [sloped,pos=0.6] {\footnotesize$\arc^\otherNode$} node [sloped,swap,pos=0.25] {\footnotesize$\beta_{\arc^\otherNode}=1$} (v);
\draw [arc] (u) to node [sloped,pos=0.5] {\footnotesize$\arc^\thirdNode$} node [sloped,swap,pos=0.75] {\footnotesize$\beta_{\arc^\thirdNode}=1$} (w);
\end{tikzpicture}
\caption{Binary decision gadget with exactly two possibilities for $\node$'s potential: $\pot^1_\node=\pot^2_\node=1$
or~$\{\pot^1_\node,\pot^2_\node\}=\{0,4\}$}
\label{fig:gadget2}
\end{figure}
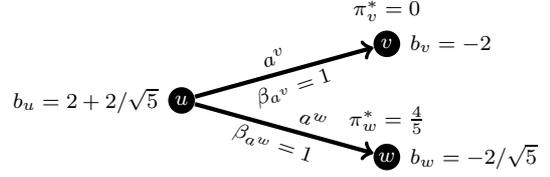

In order to satisfy the demands at nodes~$\otherNode$
and~$\thirdNode$ in a $2$-stage gas flow, the
potentials~$\pot_\node^1$ and~$\pot_\node^2$ need to satisfy the
following equations:
\begin{align}
\begin{split}
2=&\sgn(\pot_\node^1)\sqrt{\abs{\pot_\node^1}}+\sgn(\pot_\node^2)\sqrt{\abs{\pot_\node^2}},\\
\frac2{\sqrt{5}}=&\sgn(\pot_\node^1-\tfrac45)\sqrt{\abs{\pot_\node^1-\tfrac45}}\\
                  &+\sgn(\pot_\node^2-\tfrac45)\sqrt{\abs{\pot_\node^2-\tfrac45}}.
\end{split}\label{eq:binary-decision}
\end{align}
It is straightforward to verify that the only solutions (up to symmetry) to~\eqref{eq:binary-decision} are $\pot^1_\node=\pot^2_\node=1$
and~$\{\pot^1_\node,\pot^2_\node\}=\{0,4\}$.

\subsection{Reduction from Exact Cover By 3-Sets}

We prove Theorem~\ref{thm:NP-hard} via a reduction of the NP-complete problem 
Exact Cover By 3-Sets (X3C): 
the input is a finite set $X$ of cardinality~$\abs{X}=3q$, and a family of
subsets~$\mathcal{C}\subseteq 2^X$ with~$\abs{C}=3$ for all~$C\in\mathcal{C}$.
The question is whether there is a sub-family~$\mathcal{C}'\subseteq\mathcal{C}$
with~$\bigcup_{C\in\mathcal{C}'}C=X$ and~$\abs{\mathcal{C}'}=q$.

\begin{proof}[Proof of Theorem~\ref{thm:NP-hard}.]
Given an instance of X3C, we construct an instance of the $2$-stage
gas $b$-flow problem as follows. Set~$\pot_{\min} \define 0$ 
and~$\pot_{\max} \define 4$. For each element~$x\in X$, we introduce two
nodes~$\node^x$ and~$\otherNode^x$ with fixed 
potentials~$\pot^*_{\node^x} \define \frac{16}{25}$ and~$\pot^*_{\otherNode^x}=1$
(see gadget in Fig.~\ref{fig:gadget1}). 
For each~$C\in\mathcal{C}$, we introduce a binary decision node~$\node^C$
with~$\pot^1_{\node^C}=\pot^2_{\node^C}=1$
or~$\{\pot^1_{\node^C},\pot^2_{\node^C}\}=\{0,4\}$ (see gadget in
Fig.~\ref{fig:gadget2}). We say that node~$\node^C$ is \emph{off} if
$\pot^1_{\node^C}=\pot^2_{\node^C}=1$, otherwise it is \emph{on}.
Finally, for each~$C\in\mathcal{C}$ and each of the three~$x\in C$,
we introduce two arcs~$(\node^x,\node^C)$ and~$(\node^C,\otherNode^x)$.
By construction, if node~$\node^C$ is off,
\begin{align*}
\flow_{(\node^x,\node^C)}^1+\flow_{(\node^x,\node^C)}^2
&=-2\frac{\sqrt{1-\frac{16}{25}}}{\sqrt{\beta_{(\node^x,\node^C)}}}
=\frac{-\frac65}{\sqrt{\beta_{(\node^x,\node^C)}}}
\end{align*}
and $\flow_{(\node^C,\otherNode^x)}^1+\flow_{(\node^C,\otherNode^x)}^2=0$.
If node~$\node^C$ is on,
\begin{align*}
\flow_{(\node^x,\node^C)}^1+\flow_{(\node^x,\node^C)}^2
&=\frac{\frac45-\frac25\sqrt{21}}{\sqrt{\beta_{(\node^x,\node^C)}}}
\end{align*}
and $\flow_{(\node^C,\otherNode^x)}^1+\flow_{(\node^C,\otherNode^x)}^2 =(\sqrt{3}-1)/{\sqrt{\beta_{(\node^C,\otherNode^x)}}}$.
Thus, if we set $\beta_{(\node^x,\node^C)} \define(2-\tfrac25\sqrt{21})^2$, $\beta_{(\node^C,\otherNode^x)} \define (\sqrt{3}-1)^2$,
then changing the state of~$\node$ from off to on increases the
total flow along arc~$(\node^x,\node^C)$ and along 
arc~$(\node^C,\otherNode^x)$ by~$1$, leaving the flow balance at
node~$\node^C$ unchanged.

The idea of the reduction is that the demands of nodes~$\node^x$
and~$\otherNode^x$ are satisfied if and only if exactly one 
node~$\node^C$ with~$x\in C$ is on. For all~$x\in X$ and~$C\in\mathcal{C}$, let 
\begin{align*}
b_{\node^x}
&\define 1-\frac{\frac65\abs{\{C\in\mathcal{C}\mid x\in C\}}}{2-\tfrac25\sqrt{21}}, \ \
b_{\otherNode^x} \define -1,\\
b_{\node^C}&\define 3\frac{\frac65}{2-\tfrac25\sqrt{21}}+\bigl(2+2/\sqrt{5}\bigr),
\end{align*}
where the term~$2+2/\sqrt{5}$ stems from the decision
node gadget (see Fig.~\ref{fig:gadget2}). By construction of the reduction, there is a one-to-one
correspondence between feasible solutions to the X3C instance and
feasible solutions to the $2$-stage gas $b$-flow instance.
\end{proof}

\begin{remark}
In view of the fact that, due to the irrational numbers involved,
stationary and instationary gas flows can only be approximately computed
anyway, the significance of Theorem~\ref{thm:NP-hard} might seem
questionable at first glance. Notice, however, that due to continuity of
all functions involved, the gadgets in the proof are robust
toward small changes of numbers. In particular, it is even NP-hard 
to decide whether there exists an almost feasible $2$-stage gas flow
approximately fulfilling supplies and demands~$b$.
\end{remark}

It is also interesting to compare the hardness result of Theorem~\ref{thm:NP-hard} to related hardness
results for flows over time which constitute a time-dependent variant
of classical network flows; see, \eg the survey article~\cite{Skutella-Korte09}.
Most flow over time problems are only weakly NP-hard, if not polynomially
solvable. In particular, they can be solved efficiently as long as the
number of discrete time steps is polynomially bounded in the input size.
In contrast, our instationary gas $b$-flows are strongly NP-hard already
for only two time steps. This is mainly due to the non-convexity of the
square root function describing the relationship between flows and node
potentials.

\begin{remark}
We would finally like to point out that the results in this paper are meaningful beyond
the area of gas transport. The presented observations can be generalized to
potential-based flow models such as those considered in~\cite{GrossPfetschScheweEtAl2017},
as long as the function~$f$ in
$x_{(\node,\otherNode)}=\sgn(\pot_\node-\pot_\otherNode)f(|\pot_\node-\pot_\otherNode|/\beta_{(\node,\otherNode)})$
is strictly concave (for gas flows, $f(z)=\sqrt{z}$).	
\end{remark}

\paragraph{Acknowledgements.} We acknowledge funding through the DFG CRC/TRR~154,
Subprojects~A01 and A007. The last author is supported by the Einstein Foundation Berlin.

\bibliographystyle{elsarticle-num} 
\bibliography{InstationaryGasFlows}

\end{document}